\documentclass[12pt]{article}
\usepackage[numbers,sort&compress]{natbib}

\usepackage{color}

\definecolor{red}{rgb}{1,0,0}

\usepackage{t1enc}
\usepackage[latin1]{inputenc}
\usepackage[english]{babel}
\usepackage{amsmath,amsthm}
\usepackage{amsfonts}
\usepackage{latexsym}
\usepackage{graphicx}
\usepackage{txfonts}
\usepackage[natural]{xcolor}
\usepackage{lineno}

\newtheorem{theorem}{Theorem}
\newtheorem{corollary}[theorem]{Corollary}

\newtheorem{lemma}[theorem]{Lemma}

\title{Colorings v.s. list colorings of uniform hypergraphs\thanks{Supported by the
National Natural Science Foundation of China under Grant No.\,11471273 and 11561058.}}
\author
{
Wei Wang$^{\rm a,b}$,
Jianguo Qian$^{\rm a}$\thanks{Corresponding author: jgqian@xmu.edu.cn.},
Zhidan Yan$^{\rm b}$
\\
{\footnotesize$^{\rm a}$School of Mathematical Sciences, Xiamen University, Xiamen 361005, P. R. China}\\
{\footnotesize$^{\rm b}$College of Information Engineering, Tarim University, Alar 843300, P. R. China}
}
\date{}
\begin{document}
 \maketitle
\begin{abstract} Let $r$ be an integer with $r\ge 2$ and $G$ be a connected $r$-uniform hypergraph with $m$ edges. By refining the broken cycle theorem for hypergraphs, we show that if $k>\frac{m-1}{\ln(1+\sqrt{2})}\approx 1.135 (m-1)$ then the $k$-list assignment of $G$ admitting the fewest colorings is the constant list assignment. This extends the previous results of Donner, Thomassen and the current authors for graphs.\\

\end{abstract}
\section{Introduction}
\label{intro}
 For a positive integer $k$, a $k$-\emph{list assignment} of a graph or hypergraph $G=(V,E)$ is a mapping $L$ which assigns to each vertex $v$ a set $L(v)$ of $k$ permissible colors.  Given a $k$-list assignment $L$, an \emph{$L$-coloring} of $G$ is a vertex coloring in which the color of every vertex $v$ is chosen from its list $L(v)$ and every edge contains a pair of vertices with different colors. The notion of list coloring  was introduced independently by Vizing \cite{vizing1976} and by Erd\H{o}s, Rubin and Taylor \cite{erdos1979} initially for  ordinary graphs and then was extended to hypergraphs \cite{Alon2,Benzaken,Kahn,Ramamurthi,Saxton1,Saxton2}.

The list coloring for graphs has been extensively studied, much of the earlier fundamental work on which was surveyed in Alon \cite{Alon1}, Tuza  \cite{Tuza} and Kratochv\'{\i}l-Tuza-Voigt \cite{Kratochvil}. One direction of interests on list coloring focused on the evaluation or asymptotic behaviour of the \emph{list chromatic number} $\chi_l(G)$ (also called the \emph{choice number}), that is, the minimum $k$ such that $G$ has an $L$-coloring for any $k$-list assignment $L$. For a graph $G$, it is well known that  $\chi_l(G)$ can be much larger than the chromatic number  $\chi(G)$ \cite{erdos1979}. More specifically, in  \cite{Alon1} Alon showed that the list chromatic number of a graph grows with the average degree. However, this is not the case for hypergraphs. It was shown that, when $r\geq 3$, it is not true in general that the list chromatic number of  $r$-uniform hypergraphs grows with its average degree \cite{Alon2}.  Even so, it was also shown that similar property holds for many classes of hypergraphs \cite{Alon2,Haxell,Saxton2}, including all the simple uniform hypergraphs (here, a hypergraph is \emph{simple} if different edges have at most one vertex in
common) \cite{Saxton1}.

In this article we focus on other direction, that is, the number of list colorings in hypergraphs.  For a $k$-list assignment $L$ of a graph or hypergraph $G$, we use $P(G,L)$ to denote the number of $L$-colorings of $G$. Note that if  $L(v)=\{1,2,\ldots,k\}$  for every vertex $v\in V$, then an $L$-coloring is exactly an ordinary $k$-coloring and hence $P(G,L)$ agrees with the classic \emph{chromatic polynomial} $P(G,k)$. We call such $L$ a \emph{constant $k$-list assignment}. The \emph{list coloring function}, denoted by $P_l(G,k)$, is the minimum value of $P(G,L)$ over all $k$-list assignments $L$ of $G$, which was initially introduced  in \cite{kostochka1990,thomassen2009} for graphs.

For an ordinary graph $G$, Donner  in 1992 \cite{donner1992} showed that the list coloring function $P_l(G,k)$ equals  the chromatic polynomial $P(G,k)$ when $k$ is sufficiently large (compared to the number of vertices). Later in 2009, Thomassen \cite{thomassen2009} specified the `sufficiently large $k$' by `$k> |V|^{10}$'. Recently, the latter result was improved further to $k>\frac{m-1}{\ln(1+\sqrt{2})}$ by the present authors \cite{wangqianyan2017}, where $m$ is the number of edges. In contrast to ordinary graphs, there seems to be very few results on the number of the list colorings for hypergraphs. In this paper, we generalize the above result to $r$-uniform hypergraphs for any $r\ge 2$:

\begin{theorem} \label{main} For any connected $r$-uniform hypergraph $G$ with $m$ edges and $r\geq 2$,  if
\begin{equation} k>\frac{m-1}{\ln(1+\sqrt{2})}\approx 1.135 (m-1)
\end{equation}
then $P_l(G, k) = P(G, k)$ and the $k$-list assignment admitting the fewest colorings is  the constant $k$-list assignment.
\end{theorem}
To prove this theorem, we use the technique of the famous  Whitney's broken cycle theorem, by which the calculation of the chromatic polynomial of a graph $G$ can be restricted to the collection of those sets of edges which do not include any broken cycle as a subset. Whitney's broken cycle theorem was established initially for  graphs and has been extended to hypergraphs in various forms based on different ways to define a `cycle' for hypergraphs, see \cite{Dohmen01} for an example. Recently, Trinks \cite{trinks2014} gave a form for hypergraphs  by introducing a notion of `$\delta$-cycle'.

For ordinary graphs, the significance of the broken cycle theorem is that  the coefficient of $x^i$ in $P(G,k)$ is exactly the number of the sets consisting of $n-i$ edges  that contain no broken cycle. However, this is not the case for hypergraphs because two sets of edges that  do not include any broken cycle may induce different number of connected components, even if they have the same number of edges.

For this reason,  we give a refined version  of the broken cycle theorem. Our refinement is based on the form  given by Trinks in which the $\delta$-cycle provides a nice structure  for defining a broken cycle and, further, provides us an effective way to constrain the number of edges in the hypergraphs that contain no broken cycle. Indeed, the  $\delta$-cycle not only gives a nice structure for defining a `broken cycle' but, among others, gives a very nice understanding for what a cycle in a hypergraph should be.

Finally, we give a  parallel consequence of Theorem \ref{main} for improper list colorings of graphs.

\section{Broken Cycle Theorem for Hypergraphs}
In the following, all hypergraphs have no parallel edges, that is, no two edges are the same (when viewed as sets of vertices). For a set $V$ of vertices and $E$ of (hyper-) edges, we use $(V,E)$ to denote the hypergraph with vertex set $V$ and edge set $E$.

Let  $G=(V,E)$ be a hypergraph. $G$  is called \emph{trivial} if $E=\emptyset$ and is called $r$-\emph{uniform} if each edge consists of exactly $r$ vertices. We denote by $c(V,E)$ or, simply $c(G)$, the number of connected components of $G$.  A nonempty set $F\subseteq E$ is called a {\it $\delta$-cycle} of $G$ if
\begin{equation}\label{def}
c(V,F\setminus\{f\})=c(V,F)
\end{equation}
for each $f\in F$ and contains no proper nonempty subset satisfying this requirement (see \cite{trinks2014} for details). We note that in the case of graphs (2-uniform hypergraphs) the definition of $\delta$-cycles agrees with the definition of the usual cycles.

Let `<' be a fixed linear order  on the edge set $E$. A set $B\subseteq E$ is called a {\it broken cycle} (with respect to `<') if $B$ is obtained from a $\delta$-cycle by deleting its maximum edge. Define the set system
\begin{equation}
\mathcal{B}(G)=\{S\colon\, S\subseteq E~\text{and $S$ contains no broken cycle}\}.
\end{equation}
The following broken cycle theorem for hypergraphs was given by Trinks \cite{trinks2014}.
\begin{lemma}\cite{trinks2014}\label{restrinks}
\begin{equation}\label{forwhitney}
P(G,k)=\sum_{S\in \mathcal{B}(G)} (-1)^{|S|} k^{c(V,S)}.
\end{equation}
\end{lemma}
In the case of graphs, if $S\in \mathcal{B}(G)$ then  $c(V,S)$ can be simplified to $|V|-|S|$ since $(V,S)$ is a forest. For $r$-uniform hypergraphs, we will show that $c(V,S)$ is bounded above by $|V|-|S|+2-r$, by which we will give a more applicable version of the broken cycle theorem for $P(G,k)$ and $P(G,L)$, respectively.
\begin{lemma}\label{boundE}
Let $r$ be an integer with $r\ge 2$ and $G=(V,E)$ be a nontrivial $r$-uniform  hypergraph. If $G$ contains no $\delta$-cycle then $G$ contains at most $|V|-r+1$ edges.
\end{lemma}
\begin{proof}
Since $G$ is nontrivial, we have $|V|\ge r$. We prove the assertion by induction on $|V|$. If $|V|=r$ then $|E|=1$ since $G$ has no parallel edges. The assertion holds in this case. We now consider the case $|V|=n>r$.

 By the definition of $\delta$-cycle, if $c(V,E\setminus\{e\})=c(V,E)$ for each $e\in E$ then $G$ contains  a $\delta$-cycle, which contradicts the assumption that $G$ contains no $\delta$-cycle. This means that there exists an edge $e^*\in E$ such that
$c(V,E\setminus\{e^*\})\neq c(V,E)$ and, thus, $c(V,E\setminus\{e^*\})\ge 2$.

Let $G_1,G_2,\ldots,G_s$ be all the
nontrivial connected components and $G_{s+1},G_{s+2},\ldots, G_{s+t}$ be all the trivial connected components (isolated vertices) of $(V,E\setminus\{e^*\})$, where $s+t=c(V,E\setminus\{e^*\})\geq 2$. If $s=0$, then $G$ contains exactly one edge $e^*$ and hence, $|E|=1<|V|-r+1$. This means that the assertion holds in this case. We now assume that $s\geq 1$. Since $s+t\geq 2$, each $G_i$ with $i\in\{1,2,\ldots,s\}$ contains at most $|V|-1$ vertices. So by the induction hypothesis, $G_i$ contains at most $|V(G_i)|-r+1$ edges. Thus, we have
\begin{eqnarray*}
|E|&= &|E(G_1)|+|E(G_2)|+\cdots+|E(G_s)|+1\\
& \le & (|V(G_1)|-r+1)+(|V(G_2)|-r+1)+\cdots+(|V(G_s)|-r+1)+1\\
&=&(|V|-t)-(r-1)s+1\\
& \le & |V|-(r-1)s+1\\
& \le  & |V|-r+2.
 \end{eqnarray*}
Notice that the last two equalities can not hold simultaneously since otherwise $t=0$ and $s=1$, contradicting the fact that $s+t\ge 2$. Therefore, $|E|\le |V|-r+1$. The lemma follows by induction.
\end{proof}

\begin{corollary}\label{boundC}
Let $r$ be an integer with $r\ge 2$  and $G=(V,E)$ be a nontrivial $r$-uniform hypergraph with no $\delta$-cycle. Then,
\begin{equation}\label{lowup}
\max\{1,|V|-(r-1)|E|\}\le c(G)\le |V|-|E|+2-r,
\end{equation}
with both equalities holding if $r=2$ or $|E|=1$.
\end{corollary}
\begin{proof}
Since $c(G)\geq 1$, to prove  the left inequality of (\ref{lowup}), it suffices to prove that
\begin{equation}\label{low}
c(G)\ge |V|-(r-1)|E|.
\end{equation}
 When $|E|=1$ the assertion holds directly. Now let $|E|>1$. We notice that adding an edge to $G$ decreases  at most $r-1$ components. Therefore, (\ref{low})  follows by a simple induction on $|E|$.

Next we prove the right inequality of (\ref{lowup}). As $G$ is nontrivial, it has at least one nontrivial component. Let $c=c(G)$ and $G_1,G_2,\ldots,G_s$ be all nontrivial components of $G$. Then $|V(G_1)|+|V(G_2)|+\cdots+|V(G_s)|=|V|-(c-s)$. Combining this equality with Lemma \ref{boundE} leads to  \begin{eqnarray*}
|E|&= &|E(G_1)|+|E(G_2)|+\cdots+|E(G_s)|\\
   & \le & (|V(G_1)|-r+1)+|V(G_2)|-r+1)+\cdots+(|V(G_s)|-r+1)\\
   &\le  & (|V|-(c-s))+s(1-r)\\
   & =  & |V|-c+s(2-r),
 \end{eqnarray*}
implying  $c\le |V|-|E|+s(2-r)$. Since $2-r\le 0$ and $s\ge 1$, we have $s(2-r)\le 2-r$ and hence  $c\le |V|-|E|+2-r$, as desired.

One can easily verify that both equalities hold if $r=2$ or $|E|=1$. This completes the proof of the corollary.
\end{proof}

Let $G=(V,E)$ be an $r$-uniform hypergraph with $n$ vertices and let $S\in \mathcal{B}(G)$. By the definition of $\mathcal{B}(G)$, $S$ contains no broken cycle and hence, contains no $\delta$-cycle. Therefore, $|S|\le n-r+1$ by Lemma \ref{boundE}. We write
\begin{equation}
\mathcal{B}(G)=\mathcal{B}_0(G)\cup \mathcal{B}_1(G)\cup\cdots\cup\mathcal{B}_{n-r+1}(G),
\end{equation}
where
\begin{equation}
\mathcal{B}_i(G)=\{S\in \mathcal{B}(G)\colon\,|S|=i\}.
\end{equation}
By Eq. (\ref{def}) in the definition of the $\delta$-cycle, it is easy to see that each $\delta$-cycle contains at least three edges since $G$ has no parallel edges. This indicates that each broken cycle contains at least two edges. Therefore $\mathcal{B}_0(G)=\{\emptyset\}$ and $\mathcal{B}_1(G)=\{\{e\}\colon\,e\in E\}$.

Clearly, $c(V,\emptyset)=n$ and $c(V,\{e\})=n-r+1$. In other words,
\begin{equation}
c(V,S)=\begin{cases}
n,&S\in\mathcal{B}_0(G),\\
n-r+1,&
S\in\mathcal{B}_1(G).
\end{cases}
\end{equation}
Let $i\in\{1,2,\ldots,n-r+1\}$ and $S\in \mathcal{B}_i(G)$. By Corollary \ref{boundC} we have
$$\max\{1,n-(r-1)i\}\le c(V,S)\le n-i+2-r.$$
Thus, for each $i\in \{1,2,\ldots,n-r+1\}$ we can write
\begin{equation}
\mathcal{B}_i(G)=\mathcal{B}_i^\tau(G)\cup \mathcal{B}_i^{\tau+1}(G)\cup\cdots\cup\mathcal{B}_{i}^{n-i+2-r}(G),
\end{equation}
where $\tau=\max\{1,n-(r-1)i\}$ and
\begin{equation}\label{defbij}
\mathcal{B}_i^j(G)=\{S\in \mathcal{B}_i(G)\colon\,c(V,S)=j\},
\end{equation} for $j$ with $\tau\le j\le  n-i+2-r$.

Thus, the following refined version of Lemma \ref{restrinks} follows immediately.
\begin{theorem}\label{newtrinks}
\begin{equation}\label{newform}
P(G,k)=k^n+\sum_{i=1}^{n-r+1}(-1)^{i}\sum_{j=\tau}^{n-i+2-r}\sum_{S\in \mathcal{B}_i^j(G)}  k^{j},
\end{equation}
where $\tau=\max\{1,n-(r-1)i\}$.
\end{theorem}

 Let $L$ be a $k$-list assignment of $G$ and let $C=(V_1,E_1)$ be a connected subgraph of $G$, where $V_1\subseteq V, E_1\subseteq E$. We define
\begin{equation}
\beta(C,L)=\Big|\bigcap_{v\in V_1}L(v)\Big|,
\end{equation}
that is,  $\beta(C,L)$ is the number of the common colors shared by  all the vertices of $C$.

For $S\subseteq E$, let
\begin{equation}\label{defoff}
f(S)=(-1)^{|S|}\prod_{t=1}^{c(V,S)}\beta(C_t^S,L),
\end{equation}
where $C_1^S,C_2^S,\ldots,C_{c(V,S)}^S$ are all the connected components (including the isolated vertices) of the subgraph $(V,S)$ of $G$.

By the inclusion-exclusion principle,
\begin{equation}\label{fromie}
P(G,L)=\sum_{S\subseteq E}f(S).
\end{equation}

Further, it can be seen that,  for all $S\subseteq E$ and all $e\in E\setminus S$,  it holds
\begin{equation}
c(V,S)=c(V,S\cup \{e\})\ \ \Rightarrow\ \ f(S)=-f(S\cup\{e\}).
\end{equation}

\begin{lemma} (Theorem 4, \cite{trinks2014}) \label{trinks1}
Let $G=(V,E)$ be a hypergraph with a linear order < on the edge set $E$,  $\mathcal{C}$ a  set of (not necessarily all) broken cycles of $G$ and let $f(S)$ be a function to an additive  group such that for all $S\subseteq E$ and all $e\in E\setminus S$ it holds
$$c(V,S)=c(V,S\cup \{e\})\ \ \Rightarrow\ \ f(S)=-f(S\cup\{e\}).$$
Then
\begin{equation}\label{fromthe4}
\sum_{S\subseteq E}f(S)=\sum _{S\in \mathcal{B_\mathcal{C}}(G)}f(S),
\end{equation}
where $\mathcal{B}_\mathcal{C}(G)=\{S\colon\,S\subseteq E \text{~and $S$ contains no broken cycle in~} \mathcal{C}\}$.
\end{lemma}
For our purpose, we only consider the case that $\mathcal{C}$ consists of all broken cycles of $G$. Then $\mathcal{B}_\mathcal{C}(G)=\mathcal{B}(G)$ and (\ref {fromthe4}) becomes
\begin{equation}\label{fromthe4spe}
\sum_{S\subseteq E}f(S)=\sum _{S\in \mathcal{B}(G)}f(S).
\end{equation}
Combining (\ref{fromthe4spe})  with (\ref{defoff}) and (\ref{fromie}),  we have
\begin{equation}
P(G,L)=\sum_{S\in \mathcal{B}(G)}(-1)^{|S|}\prod_{t=1}^{c(V,S)}\beta(C_t^S,L).
\end{equation}

Moreover, by the same discussion as for Theorem \ref{newtrinks}, we get the following form of the broken cycle theorem for $P(G,L)$.
\begin{theorem}\label{listtrinks}
\begin{equation}\label{listnewform}
P(G,L)=k^n+\sum_{i=1}^{n-r+1}(-1)^{i}\sum_{j=\tau}^{n-i+2-r}\sum_{S\in \mathcal{B}_i^j(G)} \prod_{t=1}^{j} \beta(C_t^S,L),
\end{equation}
where $\tau=\max\{1,n-(r-1)i\}$ and $C_1^S,C_2^S,\ldots,C_{j}^S$ are all the components  of  $(V,S)$.
\end{theorem}
\section{Proof of Theorem \ref{main}}
Let $G=(V,E)$ be a connected $r$-uniform hypergraph with $n$ vertices and $m$ edges and let $L$ be a $k$-list assignment of $G$. For an edge $e\in E$, we denote by $V(e)$ the set of the vertices in $e$ and let
\begin{equation}\label{defa}
\alpha(e,L)=k-\Big|\bigcap_{v\in V(e)} L(v)\Big|.
\end{equation}
For $F\subseteq E$,  denote $V(F)=\bigcup_{e\in F}V(e)$.
\begin{lemma}\label{resbeta}
Let $W\subseteq V$ and $F\subseteq E$. If the hypergraph $C=(W,F)$ is connected then
\begin{equation}\label{inebeta}
k\ge \beta(C,L)\ge k-\sum_{e\in F}\alpha(e,L),
\end{equation}
where the right equality holds if $|W|=1$ or $|W|=r$.
\end{lemma}
\begin{proof}
As $|L(v)|=k$ for each $v\in W$, the left inequality trivially holds.
Let $p=|W|$ and $q=|F|$. We prove the right inequality by induction on $p+q$. Since $C$ is connected, we have $p=1$ or $p\ge r$. If $p=1$ or $p=r$ then the right `$\geq$' becomes `$=$' and the assertion clearly holds.

Now consider the case that $p+q>r+1$ (and hence $q\ge 2$). Since $(W,F)$ is connected, we may label the edges in $F$ as $e_1,e_2,\ldots,e_{q}$ so that the hypergraph $C_i=(V(F_i),F_i)$ is connected for each  $i\in \{1,2,\ldots,q\}$, where $F_i=\{e_1,e_2,\ldots,e_{i}\}$. This implies that
\begin{equation}
V(F_{i-1})\cap V(e_i)\neq \emptyset, ~\text{for}~2\le i\le q.
\end{equation}
Choose an arbitrary $u\in V(F_{q-1})\cap V(e_{q})$. We have
 \begin{equation}
\bigcap_{v\in V(F_{q-1})}L(v)\subseteq L(u)~~~\text{and}~\bigcap_{v\in V(e_{q})}\subseteq L(u),
\end{equation}
and hence
 \begin{equation}\label{cuplessk}
 \Big|\Big(\bigcap_{v\in V(F_{q-1})}L(v)\Big)\cup\Big(\bigcap_{v\in V(e_q)}L(v)\Big)\Big|\le |L(u)|= k.
 \end{equation}
By (\ref{cuplessk}) and the simple formula $|A\cap B|=|A|+|B|-|A\cup B|$, we obtain
\begin{eqnarray*}
\beta(C,L)& =& \Big|\bigcap_{v\in W}L(v)\Big|\\
& =&  \Big|\Big(\bigcap_{v\in V(F_{q-1})}L(v)\Big)\cap\Big(\bigcap_{v\in V(e_{q})}L(v)\Big)\Big|\\
&\ge &  \Big|\bigcap_{v\in V(F_{q-1})}L(v)\Big|+\Big|\bigcap_{v\in V(e_{q})}L(v)\Big|-k\\
&  =  & \beta(C_{q-1},L)-\alpha(e_{q},L)\\
&\ge & \Big( k-\sum_{i=1}^{q-1}\alpha(e_i,L)\Big)-\alpha(e_{q},L)\\
&  =  &k-\sum_{e\in F}\alpha(e,L),
 \end{eqnarray*}
where the last inequality follows by the induction hypothesis. This completes the proof of the lemma.
\end{proof}
We now estimate the difference between
$$\sum_{j=\tau}^{n-i+2-r}\sum_{S\in \mathcal{B}_i^j(G)}  k^{j}$$
 in (\ref{newform}) and
$$\sum_{j=\tau}^{n-i+2-r}\sum_{S\in \mathcal{B}_i^j(G)} \prod_{t=1}^{j} \beta(C_t^S,L)$$
in (\ref{listnewform}). To this end, we need the following inequality, which is obtained from the Weierstrass inequality:  $\prod_{i=1}^s(1-a_i)\ge 1-\sum_{i=1}^s a_i$, where $0\leq a_i\leq 1$ for every $i=1,2,\ldots,s$ \cite{kuang2010}.
\begin{lemma}\label{ressimple}
Let $t$ be a positive integer and let $a_1,a_2,\ldots,a_s$ be $s$ real numbers in $[0,t]$. Then
\begin{equation}\label{inesimple}
\prod_{i=1}^s(t-a_i)\ge t^s-t^{s-1}\sum_{i=1}^s a_i
\end{equation}
and the equality holds if at most one of these $a_i$'s is positive.
\end{lemma}

\begin{lemma}\label{resineBi}
For $1\le i\le n-r+1$,
\begin{equation}\label{ineBi}
0\le  \sum_{j=\tau}^{n-i+2-r}\sum_{S\in \mathcal{B}_i^j(G)}  k^{j}-\sum_{j=\tau}^{n-i+2-r}\sum_{S\in \mathcal{B}_i^j(G)} \prod_{t=1}^{j} \beta(C_t^S,L)\le k^{n-i+1-r}\binom{m-1}{i-1}\sum_{e\in E}\alpha(e,L),
\end{equation}
and the right equality holds if $i=1$.
\end{lemma}
\begin{proof}
Since $k\ge \beta(C_t^S,L)$, the left inequality in (\ref{ineBi}) is obvious.

As $\beta(C_t^S,L)\ge 0$, the right inequality in (\ref{inebeta}) implies
\begin{equation}\label{betabound}
\beta(C_t^S,L)\ge k-\min\Big\{k,\sum_{e\in E(C_t^S)} \alpha(e,L)\Big\}.
\end{equation}
Thus, by Lemma \ref{ressimple},
\begin{eqnarray}\label{diff}
k^{j}-\prod_{t=1}^{j}\beta(C_t^S,L)&\le& k^{j}-\prod_{t=1}^{j}\bigg(k-\min\bigg\{k,\sum_{e\in E(C_t^S)}\alpha(e,L)\bigg\}\bigg)\nonumber\\
& \le &  k^{j-1}\sum_{t=1}^{j}\min\bigg\{k,\sum_{e\in E(C_t^S)}\alpha(e,L)\bigg\}\nonumber\\
&\le &   k^{j-1}\sum_{t=1}^{j}\sum_{e\in E(C_t^S)}\alpha(e,L)\nonumber\\
&=& k^{j-1}\sum_{e\in S}\alpha(e,L).
 \end{eqnarray}
For $j$ with $\tau\le j\le n-i+2-r$, let
\begin{equation}
\mathcal{E}_i^j(G)=\{S\colon\,S\subseteq E, |S|=i\text{~and~} c(V,S)=j\}
\end{equation}
and for each $e\in E$, let
\begin{equation}
\mathcal{E}_{i,e}^j(G)=\{S\colon\,S\in \mathcal{E}_i^j(G)\text{~and~} e\in S \}.
\end{equation}
By (\ref{defbij}), it is clear that
\begin{equation}\label{BE}
\mathcal{B}_i^j(G)\subseteq \mathcal{E}_i^j(G).
\end{equation}
Therefore, by (\ref{diff}),
\begin{eqnarray}\label{diffsum}
\sum_{S\in \mathcal{B}_i^j(G)}k^{j}-\sum_{S\in \mathcal{B}_i^j(G)}\prod_{t=1}^{j}\beta(C_t^S,L)&\le& \sum_{S\in \mathcal{B}_i^j(G)} k^{j-1}\sum_{e\in S}\alpha(e,L)\nonumber\\
 &\le&k^{j-1} \sum_{S\in \mathcal{E}_i^j(G)} \sum_{e\in S}\alpha(e,L)\nonumber\\
  &=&k^{j-1}  \sum_{e\in E} \sum_{S\in \mathcal{E}_{i,e}^j(G)}\alpha(e,L)\nonumber\\
  &=&k^{j-1} \sum_{e\in E}|\mathcal{E}_{i,e}^j(G)|\alpha(e,L)
  \end{eqnarray}
Let
$$ \mathcal{E}_{i,e}(G)=\{S\colon\,S\subseteq E,|S|=i\text{~and~}e\in S\}.$$
It is easy to see $$|\mathcal{E}_{i,e}(G)|=\binom{m-1}{i-1}.$$ From the definition of  $\mathcal{E}_{i,e}^j(G)$, we notice that $\mathcal{E}_{i,e}^\tau(G),\ldots,\mathcal{E}_{i,e}^{n-i+2-r}(G)$ are the pairwise disjoint subsets of $\mathcal{E}_{i,e}(G)$. Therefore,
\begin{equation}
\sum_{j=\tau}^{n-i+2-r}|\mathcal{E}_{i,e}^j|\le |\mathcal{E}_{i,e}(G)|= \binom{m-1}{i-1}.
\end{equation}
Together with (\ref{diffsum}), we have
\begin{eqnarray}\label{diffsumsum}
 \sum_{j=\tau}^{n-i+2-r}\sum_{S\in \mathcal{B}_i^j(G)}  k^{j}-\sum_{j=\tau}^{n-i+2-r}\sum_{S\in \mathcal{B}_i^j(G)} \prod_{t=1}^{j} \beta(C_t^S,L)&\le&
 \sum_{j=\tau}^{n-i+2-r}k^{j-1} \sum_{e\in E}|\mathcal{E}_{i,e}^j(G)|\alpha(e,L)\nonumber\\
 &\le&k^{n-i+1-r}\sum_{j=\tau}^{n-i+2-r} \sum_{e\in E}|\mathcal{E}_{i,e}^j(G)|\alpha(e,L)\nonumber\\
 &=&k^{n-i+1-r}\sum_{e\in E}\sum_{j=\tau}^{n-i+2-r} |\mathcal{E}_{i,e}^j(G)|\alpha(e,L)\nonumber\\
 &\le&k^{n-i+1-r}\binom{m-1}{i-1}\sum_{e\in E}\alpha(e,L).
 \end{eqnarray}
This proves (\ref{ineBi}).

Finally, if $i=1$ then $\tau=n+1-r=n-i+2-r$ and, therefore,
$$\mathcal{B}_1^{n+1-r}(G)=\mathcal{B}_1(G)= \{\{e\}\colon\,e\in E\}.$$
 Thus, the middle term in (\ref{ineBi}) is reduced to
\begin{equation}\label{redu1}
\sum_{e\in E}  k^{n+1-r}-\sum_{e\in E} \prod_{t=1}^{n+1-r} \beta(C_t^{\{e\}},L).
  \end{equation}
Since $(V,\{e\})$ has exactly one nontrivial component $(V(e),\{e\})$  and $n-r$ trivial components, we obtain
\begin{equation}
\prod_{t=1}^{n+1-r} \beta(C_t^{\{e\}},L)=k^{n-r}\big|\bigcap_{v\in V(e)} L(v)\big|
\end{equation}
Thus, (\ref{redu1}) becomes
$$k^{n-r}\sum_{e\in E}\big(k-\big|\bigcap_{v\in V(e)} L(v)\big|)=k^{n-r}\sum_{e\in E}\alpha(e,L).$$
 This means that the right equality in (\ref{ineBi}) holds when $i=1$, which completes the proof of Lemma \ref{resineBi}.
\end{proof}
\noindent\textbf{Proof of Theorem \ref{main}}. Let
 \begin{equation} \label{inekrep}
k>\frac{m-1}{\ln(1+\sqrt{2})}
\end{equation}
and $L$ be an arbitrary non-constant $k$-list assignment of $G$. It suffices to show that
\begin{equation}
P(G,L)>P(G,k).
\end{equation}
Since $L$ is not constant, the connectedness of $G$ implies that there exists an edge $e^*\in E$ such that $L$ restricted on $e^*$ is not constant.
This implies $|\cap_{v\in V(e^*)}L(v)|<k$, i.e., $\alpha(e^*,L)>0$ by (\ref{defa}).

For simplicity, we write $\alpha=\sum_{e\in E}\alpha(e,L)$ and, for $i$ with $1\le i\le n-r+1$, write
\begin{equation}
f_i= \sum_{j=\tau}^{n-i+2-r}\sum_{S\in \mathcal{B}_i^j(G)}  k^{j}-\sum_{j=\tau}^{n-i+2-r}\sum_{S\in \mathcal{B}_i^j(G)} \prod_{t=1}^{j} \beta(C_t^S,L).
\end{equation}
Then $\alpha>0$ as $\alpha\ge \alpha(e^*,L)$.  By Lemma \ref{resineBi},  $f_1=\alpha k^{n-r}$ and $0\le f_i\le \alpha\binom{m-1}{i-1}k^{n-i+1-r}$ for $2\le i\le n-r+1$. By Theorem \ref{newtrinks} and Theorem \ref{listtrinks},
\begin{eqnarray}\label{diffPlP}
P(G,L)-P(G,k)&=& \sum_{i=1}^{n-r+1}(-1)^{i-1}f_i \nonumber\\
 &\ge &f_1-\sum_{\substack{2\le i\le n-r+1\\   i\text{~even}}}f_i\nonumber\\
  &=&f_1-\sum_{\substack{1\le i\le n-r\\   i\text{~odd}}}f_{i+1}\nonumber\\
 &\ge& \alpha k^{n-r}-\sum_{\substack{1\le i\le n-r\\   i\text{~odd}}}\alpha\binom{m-1}{i}k^{n-i-r}\nonumber\\
  &\ge&\alpha k^{n-r}-\sum_{\substack{1\le i\le n-r\\   i\text{~odd}}}\alpha\frac{(m-1)^{i}}{i!}k^{n-i-r}\nonumber\\
  &=&\alpha k^{n-r}\bigg(1-\sum_{\substack{1\le i\le n-r\\   i\text{~odd}}}\frac{1}{i!}\bigg(\frac{m-1}{k}\bigg)^{i}\bigg)\nonumber\\
  &\ge&\alpha k^{n-r}\bigg(1-\frac{1}{2}\bigg(\exp\bigg(\frac{m-1}{k}\bigg)-\exp\bigg(-\frac{m-1}{k}\bigg)\bigg)\bigg).
  \end{eqnarray}

Consider the function
$$\phi(x)=1-\frac{1}{2}(\exp(x)-\exp(-x)).$$
 Let $x_0=\ln(1+\sqrt{2})$. It is easy to check that $\phi(x)$ is monotone decreasing and  $\phi(x_0)=0$. By (\ref{inekrep}), we have $\frac{m-1}{k}<x_0$ and hence
 $$\phi\left(\frac{m-1}{k}\right)>0.$$
 As $\alpha>0$, it follows from (\ref{diffPlP}) that  $P(G,L)-P(G,k)> 0$.  The proof of Theorem \ref{main} is completed.

\noindent{\bf Remark}. Theorem \ref{main} can also be interpreted as the form of improper colorings of ordinary graphs. For a graph $G$ and two integers $d$ and $k$, a \emph{$d$-improper $k$-coloring} is a mapping: $V(G)\rightarrow \{1,2,\ldots,k\}$ such that the subgraph  induced by any vertex class of the same color has maximum degree at most $d$. Let $P^d(G,k)$ denote the number of all $d$-improper $k$-colorings. The \emph{$d$-improper $k$-list coloring} \cite{Kang,Havet} and the \emph{$d$-improper list coloring function} $P^d_l(G,k)$ are defined analogously.

Regarding every set  of $d+2$ vertices that induces a subgraph of maximum degree $d+1$ as a hyperedge, we get a $(d+2)$-uniform hypergraph, denoted by $G^*$, with vertex set $V(G)$. For any $k$-list assignment $L$, it can be seen that a coloring $f$ is a $d$-improper $L$-coloring of $G$ if and only if it is an  $L$-coloring of $G^*$. Thus, the following result is a parallel consequence of Theorem \ref{main}.
\begin{theorem} Let  $G$ be a simple connected graph,  $d$ a nonnegative integer and let $p$ be the number of the sets  of $d+2$ vertices that induce a subgraph of maximum degree $d+1$. If
\begin{equation} k>\frac{p-1}{\ln(1+\sqrt{2})}\approx 1.135 (p-1)
\end{equation}
then $P_l^d(G,k)=P^d(G,k)$ and the $k$-list assignment permitting the fewest colorings is the constant $k$-list assignment.
\end{theorem}

 
\end{document}